\documentclass{article}

\usepackage{enumerate}
\usepackage{amsmath,amssymb}
\usepackage{ascmac}
\usepackage{amscd}

\newtheorem{theorem}{Theorem}[section]
\newtheorem{Cor}{Corollary}[section]
\newtheorem{Prop}{Proposition}[section]
\newtheorem{Def}{Definition}[section]
\newtheorem{Lem}{Lemma}[section]
\newtheorem{Exa}{Example}[section]

\newtheorem{proof}{Proof.}

\numberwithin{equation}{section}

\title{Transcendence of digital expansions and continued fractions generated by      
a cyclic permutation and $k$-adic expansion }

\author{Eiji Miyanohara\\
Major in Pure and Applied Mathematics\\
Graduate School of Fundamental Science and Engineering\\
Waseda University\\
3-4-1 Okubo, Shinjuku\\
Tokyo 169-8555, Japan\\
E-mail: j1o9t5acrmo@fuji.waseda.jp\\}
\date{\today}

\begin{document}

\maketitle

\begin{abstract}
In this article, first we generalize the Thue-Morse sequence $(a(n))_{n=0}^\infty$ (the generalized Thue-Morse sequences) by a cyclic
permutation and $k$ -adic expansion of natural numbers, and consider the necessary-sufficient condition that it is non-periodic.
Moreover we will show that, if the generalized Thue-Morse sequence  is not periodic,  
then all equally spaced subsequences $(a(N+nl))_{n=0}^\infty$ (where $N \ge 0$ and  $l >0$) of the generalized Thue-Morse sequences are not periodic. 
Finally we apply the criterion of [ABL], [Bu$1$] on transcendental numbers, to find that , for a non periodic generalized Thue-Morse sequences taking the values on $\{0,1,\cdots,\beta-1\}$(where $\beta$ is an integer greater than $1$), the series $\sum_{n=0}^\infty a(N+nl) {\beta}^{-n-1}$ gives a transcendental number, and further that for non periodic generalized Thue-Morse sequences taking the values on positive integers, the continued fraction $[0:a(N), a(N+l),\cdots,a(N+nl ), \cdots]$ gives a transcendental number, too.
\end{abstract}

\begin{center}\section{Introduction}\label{s:1}\end{center}
 In this paper we find many transcendental real numbers
defined by using digit counting: Let $k$ be an integer greater than $1$. We define  the $k$-adic expansion of natural number $n$ as follows,
\begin{align}
n=\sum_{q=1}^{finite}  s_{n,q} k^{w_n(q)},
\end{align}
where $1\le s_{n,q}\le k-1, w_n(q+1)>w_n(q)\ge 0$.
 Let $e_s(n)$ denote the number of counting of $s$ $($where $s \in \{1,\cdots,k-1\}$ $)$ in the
base $k$ representation of $n$. For an integer $L$ greater than $1$,
we define the sequence $(e_s^L(n))_{n=0}^\infty$ by
\begin{align}
 e_s^L(n) \equiv e_s(n) \quad \pmod{L},
\end{align}
where $0\le e_s^L(n)\le L-1$, $e_s(0)=0$. Specifically $(e_1^2(n))_{n=0}^\infty$ (where $k=2$) is known as the Thue-Morse sequence.
\newline Now we introduce a new sequence as follows. Let $K$ be a map, 
\[
     K : \{1,\dots, k-1 \} \longrightarrow \{0,1, \dots, L-1 \}.
\]
We put $(a(n))_{n=0}^\infty$ by
\[
     a(n) \equiv \sum_{s=1}^{k-1} K(s)e_s^L(n) \pmod{L},
\]
where $0\le a(n)\le L-1$.
Then the authors [MM], [AB1] proved the following result.
\begin{theorem}[MM-AB1]\label{}
Let $\beta \ge L$ be an integer. Then $\sum_{n=0}^\infty \frac{a(n)}{ {\beta}^{n+1}}$ is a transcendental number unless 
\begin{align}
               s K(1) \equiv K(s) \pmod{L}\; \;  for\;all\; 1\le  s \le k-1 \; \; and \; \;   K(k-1)\equiv 0 \pmod{L}.
\end{align}
\end{theorem}
Now we define the generalized Thue-Morse sequences as follows: 
Let  $d(n; s k^y)$ (where $s \in \{1,\cdots,k-1\}$ and $y$ is a non negative integer)  be $1$ or $0$, $d(n; s k^y)=1 $ if and only if there exists an integer
$q$ such that $s_{n,q} k^{w_n(q)}= s k^y$. Let $\kappa$ be a map, 
\[
     \kappa : \{1,\dots, k-1 \} \times \mathbb{N} \longrightarrow \{0,1, \dots, L-1 \}.
\]
where $\mathbb{N}$ denotes the set of non-negative integers (or, natural numbers).
We put  $(a(n))_{n=0}^\infty$ by
\begin{align}
    a(n) \equiv \sum_{y=0}^\infty \sum_{s=1}^{k-1} \kappa (s, y)d(n; s k^y) \pmod{L},
\end{align}
 where $0\le a(n)\le L-1$, $a(0)=0$. In this article we will call $(a(n))_{n=0}^\infty$ the generalized Thue-Morse sequences of type $(L,k,\kappa)$.
In this article, we generalize Theorem $1.1$ as follows.
\begin{theorem}
Let $\beta \ge L$ be an integer. Then $\sum_{n=0}^\infty \frac{a(N+nl)}{ {\beta}^{n+1}}$ $($ where $N \ge 0$ and  $l >0$ $)$ is a transcendental number unless there exists an integer $A$ such that
\begin{align}
\kappa(s,A+y) \equiv \kappa (1,A) s k^y    \pmod{L},
\end{align}
for all $1\le s \le k-1$ and for all $ y \in \mathbb{N}$.
\end{theorem}
The proof of Theorem $1.2$ rests on the results in section $3$ (especially Lemma $3.1$) and the transcendence criterion of [ABL].

 \; \;This paper is organized as follows: In section \ref{sec:basic}, we review the basic concepts about the 
periodicity of sequences, and give the formally definition of the generalized Thue-Morse sequences.
For a sequence $(a(n))_{n=0}^\infty$, we set its generating function $g(z) \in \mathbb{C}[[z]]$ by 
\[
      g(z):=\sum_{n=0}^\infty a(n) z^n. 
\]
For a generalized Thue-Morse sequence, one can prove that the generating function is convergent, 
and that it has an infinite product expansion. In section \ref{sec:main result}, first we prove
the key lemma on the $k$-adic expansion of natural numbers. Then we will use this lemma and 
the infinite product expansion of the generating function of 
a generalized Thue-Morse sequence to prove a necessary-sufficient condition for the non-periodicity of
the sequence. Moreover we prove that if the sequence is not periodic,  
then all equally spaced subsequences $(a(N+nl))_{n=0}^\infty$ (where $N \ge 0$ and  $l >0$) of the sequence are not periodic.  In section $4$, first we  introduce the concept of stammering sequence ([AB1], [Bu$1$])
and the combinatorial transcendence criterion of [ABL], [Bu$1$]. Then applying combinatorial transcendence criterion of [ABL], [Bu$1$] to the generalized non-periodic Thue-Morse sequence$(a(n))_{n=0}^\infty$ which take the values on $\{0,1,,\cdots,\beta-1\}$ where $\beta$ is an integer greater than $1$ (resp. which take the values on bounded positive integers), we show that $\sum_{n=0}^\infty a(N+nl){\beta}^{-n-1}$ are transcendental numbers. (resp. the continued fraction $[0:a(N), a(N+l)\cdots,a(N+nl ) \cdots]$ are transcendental numbers.) These results include Theorem $1.2$. In section $5$, first we consider the necessary-sufficient condition that a generalized Thue-Morse sequence of  is $k$-automatic sequence.  Then we can find many transcendental numbers whose irrationality exponent is finite in all equally spaced subsequences of the corresponding generalized Thue-Morse sequence
by result of [AC]. Moreover we consider transcendency of the value at algebraic point of the generating function $\sum_{n=0}^\infty a(N+nl)z^{-n-1}$ by [Bec].
\begin{center}\subsection*{Acknowledgments}\end{center}
I would like to thank Hajime Kaneko for useful advice and discussions.
I  would also like to thank Professor Jeffrey Shallit  for useful suggestion about references.
I would also like to thank Professor Kimio Ueno for providing considerable
support, and the members of Prof. Uenofs laboratory for their helpful and useful comments.
\newline
\newline
\newline
\section{Definition of the generalized Thue-Morse sequences and their generating functions}\label{sec:basic}

 Let $(a(n))_{n=0}^\infty$ be a sequence with values in $\mathbb{C}$. We say that $(a(n))_{n=0}^\infty$ 
is ultimately periodic  if there  exist non negative integers $N$ and $l(0<l )$ such that 
\begin{align}\label{def:period}
    a(n)=a(n+l)  \qquad (\forall n \ge N). 
\end{align}
An equally spaced subsequence of $(a(n))_{n=0}^\infty$ 
is defined to be a subsequence such as $(a(N+tl))_{t=0}^\infty$, where $N \ge 0$ and  $l >0$. 
\begin{Def}
Let $(a(n))_{n=0}^\infty$ be a sequence with values in $\mathbb{C}$. The sequence $(a(n))_{n=0}^\infty$ 
is called  almost everywhere non periodic if
all equally spaced  subsequences of $(a(n))_{n=0}^\infty$ do not take on only one value.

\end{Def}
 \;  Now we show some lemmas about the almost everywhere non periodic sequences.
\begin{Lem}
Let $(a(n))_{n=0}^\infty$ be almost everywhere non periodic. Then $(a(n))_{n=0}^\infty$ is not ultimately periodic.
\end{Lem}

\begin{proof}
We  show in contraposition. Assume that $(a(n))_{n=0}^\infty$ is ultimately periodic. 
From the definition of almost everywhere non periodic, we see that there exist non negative integers $N$ and $l(0<l)$ such that 
\begin{align}\label{}
            a(n)=a(n+l),   \qquad (\forall n \ge N). 
\end{align}
From this, it follows that the equally spaced subsequence $(a(N+tl))_{t=0}^\infty$ takes on only one value. This completes the proof.
\hfill
$\square$
\end{proof}

\begin{Lem}
Let $(a(n))_{n=0}^\infty$ be almost everywhere non periodic. Then all equally spaced subsequences of this sequence are almost everywhere non periodic.
\end{Lem}
\begin{proof}
We show in contraposition. If $(a(N+tl))_{t=0}^\infty$ is not almost everywhere non periodic,
then there exist non negative integers $k$ and $J(0<J)$ such that
$(a(N+kl+mJl))_{m=0}^\infty$ takes on only one value.
$(a(N+kl+mJl))_{m=0}^\infty$ is an equally spaced subsequence of $(a(n))_{n=0}^\infty$, too.
Then $(a(n))_{n=0}^\infty$ is not almost everywhere non periodic.
This completes the proof.
\hfill
$\square$
\end{proof}
\begin{Cor}
$(a(n))_{n=0}^\infty$ is almost everywhere non periodic 
if and only if all equally spaced subsequences of this sequence are not 
ultimately periodic.
\end{Cor}
\begin{proof}
We assume that $(a(n))_{n=0}^\infty$ is almost everywhere non periodic.
By Lemma $2.1,\; 2.2 $, we see that all equally spaced subsequences of this sequence are not 
periodic.

 We show the converse of collorally in contraposition. Assume $(a(n))_{n=0}^\infty$ is not almost everywhere non periodic.
Then there exist non negative integers $N$ and $l(0<l)$ such that
$(a(N+tl))_{t=0}^\infty$ takes on only one value. This sequence is ultimately periodic.
This completes the proof.
\hfill
$\square$
\end{proof}
Now we give the formally definition of the generalized Thue-Morse sequences.
\begin{Def}
Let $L$ be an integer  greater than $1$ and $a_0,a_1, \cdots ,a_{L-1}$ different $L$ complex numbers. 
Define a morphism $f$ from $\{a_0,a_1 \cdots a_{L-1}\}$ to $\{a_0,a_1 \cdots a_{L-1}\}$ as follows,
\begin{align}\label{}
f(a_i) = a_{i+1},  
\end{align}
where the indice $i$ is defined$\mod{L}$.
Let $f^j$ be the $j$ times composed mapping of $f$ and $f^0$ an identity mapping. 
Let $A$ and $B$ be finite words on $\{a_0,a_1, \cdots ,a_{L-1}\}$,
and we denote by $AB$ concatenation of $A$ and $B$.

 Let $A_0=a_0$.
Let $k$ be an integer greater than $1$, $\mathbb{N}$ the set of non negative integers. Let $\kappa$ be 
a map $\kappa$:$\{1, \cdots  ,p-1\} \times\mathbb{N} \rightarrow\ \{0, \cdots ,L-1\}$. We define $W_m$,
a space of words, by
\begin{align}
W_{m}:=\{ a_{i_1} a_{i_2} \cdots a_{i_m} \; | a_{i_1}, a_{i_2}, \cdots ,a_{i_m}\in \{a_0,a_1 \cdots a_{L-1}\} \}.
\end{align}
Define $A_{n+1}\in W_{k^{n+1}}$ recursively by
\begin{align}
     A_{n+1}:=A_n f^{\kappa(1,n)}(A_n)\cdots \cdots f^{\kappa (k-1,n)}(A_n).
\end{align}
and denote the limit of $A_n$ by    
\begin{align}
 A_{\infty}:=\lim_{n \to \infty} A_n. 
\end{align}
We call $A_{\infty}$  the generalized Thue-Morse sequence of type $(L,k,\kappa)$. We call it $(L,k,\kappa)$-TM sequence for short.
\end{Def}
\begin{Exa}  
Let $L=2$, $a_0=0$, $a_1=1$, $\kappa(1,n)=1$ \; for all $n$. Then the $(2,2,1)$-TM sequence is as follows,
\begin{align*}
         &\qquad \qquad A_0=0,\; A_1=01,\;  A_2=0110,\;   A_3=01101001,\\
         & A_{\infty}=0110100110010110100101100110100110010110011010010110100110\cdots.
\end{align*}
This is the Thue-Morse sequence.
\end{Exa} 

Let $(a(n))_{n=0}^\infty$ be a sequence with values in $\mathbb{C}$. The generating function of 
$(a(n))_{n=0}^\infty$ is a formal 
power series $g(z) \in \mathbb{C}[[z]]$ defined by 
\[
g(z):=\sum_{n=0}^\infty a(n) z^n. 
\]
The following lemma clarifies the meaning of $(L,k,\kappa)$-TM sequence and will be used in the next section.
\begin{Lem} 
Let  $A_{\infty}=((b(n))_{n=0}^\infty$ be a $(L,k,\kappa)$-TM sequene with $a_j=\exp \frac{2\pi \sqrt{-1} j}{L}$ $($where $0\le \; j \le L-1$ $)$.
Let $ G_{A_{\infty}}(z)$ be the generating function of $((b(n))_{n=0}^\infty$,
\[
 G_{A_{\infty}}(z):=\sum_{n=0}^\infty a(n) z^n. 
\]
Then $ G_{A_{\infty}}(z)$ has the infinite product on $|z|<1$ as follows,
\begin{align}
      G_{A_{\infty}}(z)=\prod_{y=0}^\infty (1+\sum_{s=1}^{k-1}\exp \frac{2\pi \sqrt{-1} \kappa(s,y) } {L}z^{s k^y}).   
\end{align}
\end{Lem}@
\begin{proof}
From the assumption that $a_j=\exp \frac{2\pi \sqrt{-1} j}{L}$, we have
\begin{align} 
f(a_j)=\exp \frac{2\pi \sqrt{-1} }{L} a_j.
\end{align}  
By the fact $(L,k,\kappa)$-TM sequence takes only finite values and the Cauchy-Hadamard theorem, 
we see that $G_{A_{\infty}}(z)$ and $\prod_{y=0}^\infty(1+\sum_{s=1}^{k-1}\exp \frac{2\pi \sqrt{-1} \kappa(s,y) } {L}z^{s k^y})$ converge absolutely on the unit disk.
Let $G_{A_n}(z)$ be the generating function of $A_n$  $($We identifies $A_n 0\cdots0\cdots $ with $A_n$ $)$.

 First, we will show $(2.9)$ by induction on $n$,
\begin{align}
         G_{A_n}(z)=\prod_{y=0}^{n-1}(1+\sum_{s=1}^{k-1}\exp \frac{2\pi \sqrt{-1} \kappa(s,y)}{L}z^{s k^y}).  
\end{align} 
First we check the case n=1. By definition of $A_1$, we have
\begin{align}
G_{A_1}(z)=1+\sum_{s=1}^{k-1}\exp \frac{2\pi \sqrt{-1} \kappa(s,0)}{L}z^s.
\end{align}
Thus the case n=1 is true. By the induction hypothesis we may assume that,
\begin{align}
             G_{A_j}(z)=\prod_{y=0}^{j-1}(1+\sum_{s=1}^{k-1}\exp \frac{2\pi \sqrt{-1} \kappa(s,y)}{L}z^{s k^y}).
\end{align}
Then we have,
\begin{align}
            G_{A_{j+1}}(z)=g_{A_j}(z)+\sum_{s=1}^{k-1} g_{f^{\kappa(s,j)}(A_j)}(z)z^{s k^j}.
\end{align}
On the other hands we have,
\begin{align}
G_{f^{\kappa(s,j)}(A_j)}(z)=\exp \frac{2\pi \sqrt{-1} \kappa(s,j)}{L}G_{A_j}(z).
\end{align}
From $(2.11)$-$(2.13)$, we get
\begin{align}
          G_{A_{j+1}}(z)=G_{A_j}(z)(1+\sum_{s=1}^{k-1}\exp \frac{2\pi \sqrt{-1} \kappa(s,y)}{L}z^{s k^j})=\prod_{y=0}^j(1+\sum_{s=1}^{k-1}\exp \frac{2\pi \sqrt{-1} \kappa(s,y)}{L}z^{s k^y}).
\end{align}
This completes the proof of $(2.9)$.  
Next we  will compare the coefficients of both sides $z^j$ of $(2.7)$. The coefficient of $z^j$ in righthand side of $(2.7)$ is 
determined by $G_{A_N}(z)$ for large enough $N$.
By definition of $A_{\infty}$, we see that the prefix $p^N$ words of $A_{\infty}$ is $A_N$.
By the fact mentioned above and $(2.9)$, the coefficients of both sides $z^j$ of $(2.7)$ coincide. This completes the proof.
\hfill
$\square$
\end{proof}
\begin{Prop}
Let $A_{\infty}=((b(n))_{n=0}^\infty$ be a $(L,k,\kappa)$-TM sequene with $a_j=\exp \frac{2\pi \sqrt{-1} j}{L}$ $($where $0\le \; j \le L-1$ $)$. Let $(a(n))_{n=0}^\infty$ be a sequence defined by $(1.4)$.
Then 
\begin{align}
\frac{L}{2\pi \sqrt{-1}} \log b(n)\equiv a(n) \; \pmod{L}.
\end{align}
\end{Prop}
\begin{proof}
Let the $k$-adic expansion of $n$ be as follows,  
\begin{align}
n=\sum_{q=1}^{n(k)} s_{n,q} k^{w_n(q)},
\end{align}
 where  $1\le s_{n,q}\le k-1$, $0 \le w_n(q)<w_n(q+1)$.
By the fact uniqueness of the $k$-adic expansion and Lemma $2.3$, we have
 \begin{eqnarray}
&b(n)=\prod_{q=1}^{n(k)} \exp \frac{2\pi \sqrt{-1} \kappa (s_{n,q},w_n(q))}{L}=\exp \frac{2\pi \sqrt{-1} (\sum_{q=1}^{n(k)}\kappa (s_{n,q},w_n(q)))}{L}\nonumber \\%
&=\exp \frac{2\pi \sqrt{-1} (\sum_{q=1}^{n(k)}\kappa (s_{n,q},w_n(q))\pmod{L})}{L}.
\end{eqnarray}
By $(2.16)$, $(2.17)$ and definition of  $a(n)$, we get
\begin{align}
\frac{L}{2\pi \sqrt{-1}} \log b(n)\equiv a(n) \; \pmod{L}.
\end{align}
This completes the proof.
\end{proof}
\begin{Def}Let $(a(n))_{n=0}^\infty$ be a sequence with values in $\mathbb{C}$.
Let $g(z)$ be the $generating$ $function$ of $(a(n))_{n=0}^\infty$.
We say $(a(n))_{n=0}^\infty$ is a $k$-adic expansion sequence 
if the $g(z)$ has the following infinite product expansion for an integer $k$ greater than $1$
and $t_{s,j} \neq0$ for all $1\le  s\le k-1$ and for all $j\in\mathbb{N}$,
\begin{align}
g(z)=\prod_{y=0}^\infty (1+\sum_{s=1}^{k-1} t_{s,y} z^{s k^y}).
\end{align}
\end{Def}

\begin{center}\section{The necessary-sufficient condition for the non-periodicity of
a generalized Thue-Morse sequence}
\label{sec:main result}\end{center}
\;First, we show the following key lemma about the $k$-adic expansion of natural numbers.

\begin{Lem} 
Let $k>1$ and $l>0$ be integers.
Let $t$ be a non negative integer. Then there exists an integer $x$ such that,
\begin{align}
xl=\sum_{q=1}^{finite}  s_{xl,q} k^{w_{xl}(q)},  
\end{align}
where $s_{{xl},1}=1, w_{xl}(2)-w_{xl}(1)>t,w_{xl}(q+1)>w_{xl}(q)\ge 0$.
\newline
Moreover let $t^{'}$ be other non negative integer. Then there exists an integer $X$ such that,
\begin{align}
Xl=\sum_{q=1}^{finite}  s_{Xl,q} k^{w_Xl(q)},
\end{align}
where $s_{{Xl},1}=1, w_{Xl}(2)-w_{Xl}(1)>t^{'}, w_{Xl}(q+1)>w_{Xl}(q)\ge 0, w_{xl}(1)=w_{Xl}(1)$.
\end{Lem}
\begin{proof}  
Assume the factorization into prime factors of $k$ as follows,
\begin{align}
k=\prod_{t=1}^k {p_t}^{y_t},
\end{align}
where $p_t \;(1\le t \le k\;)$ are $k$ prime numbers and $y_t\;(1\le t \le k\;)$ are $k$ positive integers.
Let $l$ be represented as follows,
\begin{align}
l=G\prod_{u=1}^n {p_{t_u}}^{x_u},
\end{align}
where $G$ and $k$ are coprime, ${p_{t_u}} \in \{p_t|1\le t \le k \}$ and $x_u$ are $n$ positive integers.
By the fact $G$ and $k$ are coprime, then there exist integers $D$ and $E$ such that
\begin{align}
DG=1-k^{t+1}E.
\end{align}
Let
$$F:= max  \{A  |  x_u=y_{t_u} A  +H,  \;  0 \leq H  < y_{t_u},   \; 1 \leq u  \leq n \}.$$
From the definition of $F$, we see that $k^{F+1}  \prod_{u=1}^n {p_{t_u}}^{-x_u}$ is a natural number. Then we have
\begin{align}
lD^2G k^{F+1} \prod_{u=1}^n {p_{t_u}}^{-x_u}=k^{F+1} D^2G^2.
\end{align}
On the other hands, by $(3.5)$, we have
\begin{align}
D^2G^2=1+k^{t+1}E(k^{t+1}E-2).
\end{align}
Then we see that $E(k^{t+1}E-2)$ is a natural number.
Since the $k$-adic expansion of $E(k^{t+1}E-2)$, if $E(k^{t+1}E-2)>0$, then $k^{F+1}D^2G^2$ satisfies lemma .
If $E(k^{t+1}E-2)=0$, then $G=1$. $k^{F+1}(1+k^{t+1})$ satisfies lemma.

Since $F+1$ independent of $t$, then the second claim is trivial. This completes the proof.
\hfill
$\square$
\end{proof}
 We will show almost everywhere nonperiodic result for $k$-$adic$ expansion sequences by previous lemma.  
\begin{Prop}
Let $A_{\infty}=(a(n))_{n=0}^\infty$ be a $k$-$adic$ expansion sequence and $G_{A_{\infty}}(z)$ 
the generating function of this sequence. 
If there exists an equally spaced subsequence of $(a(n))_{n=0}^\infty$ which is periodic,
then there exist a non negative integer $A$ and a constant $h$ which satisfy
\begin{align}
G_{A_{\infty}}(z)=(\sum_{n=0}^{k^A -1} a(n) z^n) \prod_{y=0}^\infty(1+\sum_{s=1}^{k-1} h^{s k^y} z^{s k^{A+y}}).
\end{align}
\end{Prop}
\begin{proof}
Let $n$ and $m$ be two natural numbers and their $k$-adic expansions are as follows,
\begin{align}
n=\sum_{q}^{finite} s_{n,q} k^{w_n(q)}, m=\sum_{p}^{finite} s_{m,p} k^{w_m(p)}
\end{align}
where $1\le s_{n,q}, s_{m,p}\le k-1$, $w_n(q+1)>w_n(q)\ge 0$, and $w_m(p+1)>w_m(p)\ge 0$.
From the definition of $k$-adic expansion sequence and the uniquness of $k$-adic expansion of natural numbers, we see that if $w_n(q)\neq w_n(p)$ for all pairs $(q,p)$,
then
\begin{align}
a(n+m)=a(n)a(m).
\end{align}
 Assume there exists an equally spaced subsequence of $(a(n))_{n=0}^\infty$ which is periodic.
By Corollary $2.1$, $(a(n))_{n=0}^\infty$ is not almost everywhere nonperiodic. 
Then there exist non negative integers $N$ and $l>0$ such that 
\begin{align}
 a(N)=a(N+tl)  \qquad   (\forall  t \in \mathbb{N}).
 \end{align}
Let the $k$-adic expansion of $N$ be as follows, 
\begin{align}
N=\sum_{q=1}^{N(k)} s_{N,q} k^{w_N(q)}   \qquad where\;  1\le s_{N,q}\le k-1, 0 \le w_N(q)<w_N(q+1).
\end{align}
By the definition of $k$-adic expansion sequence and $(3.10)$, we have
\begin{align}
        &a(N)=a(N+k^r t l)=a(N)a(k^r t l) \qquad (\forall  r > w_N(N(k))).
        \\&  \qquad   \qquad  \qquad    \qquad a(N) \neq 0.
\end{align}        
From $(3.13)$ and $(3.14)$,
we get
\begin{align}
         a(k^r t l) = 1\qquad (\forall  r > w_N(N(k))).
\end{align}
By Lemma $3.1$, we see that there exists an integer  $x$ greater than zero such that
\begin{align}
x l=\sum_{q=1}^{x l(k)} s_{xl,q} k^{w_{xl}(q)}.
\end{align}
where $s_{xl,1}=1$ and $w_{xl}(2)-w_{xl}(1)>1$.
\newline
Moreover by Lemma $3.1$, we see that there exists an integer $X$ greater than zero such that
\begin{align}
X l=\sum_{q=1}^{X l(k)} s_{{Xl}_q} k^{w_{Xl}(q)}.
\end{align}
where $s_{Xl,1}=1$, $w_{Xl}(2)-w_{Xl}(1)>w_{xl}(xl(p))$ and $w_{Xl}(1)=w_{xl}(1)$.

Let $xl k^{-w_{xl}(1)}$ $($ $Xl k^{-w_{xl}(1)}$ $)$ be replaced by $xl$ $($ respectively $Xl$ $)$.
Let  $r$ be any integer greater than $w(N(k))+w_{xl}(1)$ and $s$ any integer in $\{1,\cdots, k-1\}$.

By the definition of $Xl$ and $(3.10)$, we have
\begin{align}
a(k^r sX l)=a(sk^r )a(k^rsXl-sk^r).
\end{align}
From $(3.10)$ and $(3.15)$, we get
\begin{align}
&1=a(k^r x l),
\\&1=a(k^r sX l),
\\&1=a(k^r x l+k^r sX l).
\end{align}
By $(3.10)$, $(3.18)$-$(3.21)$, the definitions of $xl$ and $Xl$, we have
\begin{align}
    &a(k^r)a(k^r x l-k^r)=1,
    \\&a(lsk^r)a(sX lk^r-s k^r)=1,
     \\&a(k^r (s+1))a(x l k^r-k^r)a(sX lk^r-s k^r)=1.
\end{align}
From $(3.22)$-$(3.24)$, we get
\begin{align}
a(k^r (s+1))=a(k^r)a(k^r s).
\end{align}
Let $h:=a(k^{w(N(k))+w_{xl}(1)+1})$ and using the same notation of Definition $2.3$.

By $(3.10)$, we have
\begin{align}
a(sk^y)=t_{s,y},
\end{align}
for all $1\le s\le k-1$ and for all $y \in \mathbb{N}$.
\newline
By $(3.14)$, $(3.25)$, $(3.26)$ and inductively computation, we get the following relations,
\begin{align}
t_{s,w(N(k))+w_{xl}(1)+1+y}=h^{s k^y},
\end{align}
for all $1\le s\le k-1$ and for all $y \in \mathbb{N}$.
Since the definition of $k$-adic expansion sequence, this completes the proof.
\hfill
$\square$
\end{proof}
\begin{theorem} Let $A_{\infty}=(a(n))_{n=0}^\infty$ be a $(L,k,\kappa)$-TM sequence. Then $A_{\infty}=(a(n))_{n=0}^\infty$ is ultimately periodic if and only 
if there exists an integer $A$ such that
\begin{align}
\kappa(s,A+y) \equiv \kappa(1,A) s p^y    \pmod{L},
\end{align}
for all $1\le  s \le k-1$ and for all $ y \in \mathbb{N}$.

 Moreover if $(L,k,\kappa)$-TM sequence is not ultimately periodic,  
then all equally spaced subsequences of $(L,k,\kappa)$-TM sequence are not ultimately periodic.

\end{theorem}
\begin{proof}
Assume without loss of generality that $A_{\infty}=(a(n))_{n=0}^\infty$ is a $(L,k,\kappa)$-TM sequence with $a_j=\exp \frac{2\pi \sqrt{-1} j}{L}$ $($where $0\le \; j \le L-1$ $)$.
From this assumption and Lemma $2.3$, then $(a(n))_{n=0}^\infty$ is the $k$-$adic$ expansion sequence.
By the fact mentioned above and previous Proposition, we see that $(3.28)$ is the necessary condition.

 We will show converse. Let $G_{A_{\infty}}(z)$ 
be the generating function of $(a(n))_{n=0}^\infty$.
We use the same notation of Definition $2.3$. Assume $(a(n))_{n=0}^\infty$ satisfies $(3.28)$, then there exists a non negative integer $A$ such that 
\begin{align}
t_{s,A+y}=h^{s k^y}  \qquad (\forall y \in \mathbb{N}).  
\end{align}
Then $G_{A_{\infty}}(z)$ has the infinite product expansion as follows,
\begin{align}
G_{A_{\infty}}(z)=(\sum_{n=0}^{k^A -1} b(n) z^n) \prod_{y=0}^\infty(1+\sum_{s=1}^{k-1} (h z^{k^A})^{s k^y}).
\end{align}
Let $Z=h z^{k^A}$. By the fact $h$ is L-th root of $1$ and Lemma $2.3$ when $\kappa$ is zero map, we have
\begin{align} 
             \prod_{y=0}^\infty(1+\sum_{s=1}^{k-1} Z^{s k^y})=\sum_{n=0}^\infty Z^n  \qquad on \; |Z|<1.
\end{align}
Put $G(z)=\sum_{n=0}^{k^A -1} a(n) z^n$. From $(3.30)$ and $(3.31)$, we get
\begin{align}
              G_{A_{\infty}}(z)=G(z)(\sum_{n=0}^\infty (h z^{k^A})^n). 
\end{align}
Since $h$ is L-th root of $1$ and $(3.32)$, we have
\begin{align}
              G_{A_{\infty}}(z)=(G(z) (\sum_{n=0}^{L-1} (h z^{k^A})^n))(1+\sum_{s=1}^\infty z^{s L k^A})=\frac{G(z) (\sum_{n=0}^{L-1} (h z^{k^A})^n)}{1-z^{L k^A}}.
\end{align}
By the fact the degree of G(z) is $k^A -1$ and $(3.33)$,
then we see that the sequence  $(a(n))_{n=0}^\infty$ which satisfies $(3.28)$ has period $Lk^A$.

 By the fact mentioned above and Proposition $3.1$,
we see that if $(L,k,\kappa)$-TM sequence is not ultimately periodic,  
then all equally spaced sequences of $(L,k,\kappa)$-TM sequence are not ultimately periodic. This completes the proof.
\hfill
$\square$
\end{proof}
If $(L,k,\kappa)$-TM sequence independent of $n$,
then $(\kappa(1),\kappa(2),\cdots ,\kappa(k-1))$-$L$ denote $(L,k,\kappa)$-TM sequence. The weak version of the following corollary can be found as Theorem $2$ in [MM] (see also [AS2], [Fr]).
\begin{Cor}$(\kappa(1),\kappa(2),\cdots ,\kappa(k-1))$-$L$ is periodic if and only if 
  $\kappa(s)$(for all $1\le s\le k-1)$ which satisfies
\begin{align}
             s \kappa(1) \equiv \kappa(s) ,\kappa(k-1)\equiv 0 \pmod{L}.
\end{align}
Moreover if $(\kappa(1),\kappa(2),\cdots ,\kappa(k-1))$-$L$ is not periodic, 
then all equally spaced subsequences of $(\kappa(1),\kappa(2),\cdots ,\kappa(k-1))-L$
are not periodic.
\end{Cor}
\begin{proof}
By Theorem $3.1$, then we see that necessary-sufficient condition for the periodicity of $(\kappa(1),\kappa(2),\cdots ,\kappa(k-1))$-$L$ is the following relations,
\begin{align}
             \kappa(1,A+1)     \equiv \kappa(1,A)k   \pmod{L},
             \kappa(k-1)\equiv(k-1)\kappa(1) \equiv 0 \pmod{L}.
\end{align}
This completes the proof.
\hfill
$\square$ 
\end{proof}
\begin{center}\section{Transcendence results of the generalized Thue-Morse sequences}\end{center}
The authors [ABL] introduced the new class of sequences as follows:
For any positive number $y$, we define that $\lfloor y\rfloor $ and $\lceil y\rceil $ denote the floor and ceiling functions.
Let $W$ be finite word on $\{a_0,a_1 \cdots a_{L-1}\}$. Let $|W|$ be length of $W$. For any positive number $x$, we denote by $W^x$ the word $W^{ \lfloor x \rfloor} W^`$, where
$W^`$ is prefix of $W$ of length $\lceil(x-\lfloor x\rfloor)|W|\rceil$. 
\begin{Def}
$(a(n))_{n=0}^\infty$ is said to be a stammering sequence if $(a(n))_{n=0}^\infty$ satisfies the following conditions,

$(1)$ $(a(n))_{n=0}^\infty$ is a non periodic sequence.

$(2)$ There exist two sequences of finite words $(U_m)_{m\ge 1}$, $(V_m)_{m\ge 1}$ such that,

$(A)$ There exist a real number $w >1$ independent of $n$ such that, the word
$U_m {V_m}^w$ is a prefix of the word $(a(n))_{n=0}^\infty$,

$(B)$ $\lim_{m \to \infty} |U_m|/|V_m|\;< \;+\infty$,

$(C)$  $\lim_{m \to \infty} |V_m|=+\infty$.
\end{Def}
Let $(a(n))_{n=0}^\infty$ be a sequence of positive integers. We put the continued fractions as follows,
\begin{align}
 [0:a(0),a(1),\cdots, a(n)\cdots ]:=0+\cfrac{1}{a(0)+
                                                 \cfrac{1}{a(1)+
                                                 \cfrac{1}{\cdots +
                                                 \cfrac{1}{a(n)+
                                                 \cfrac{1}{
                                                 \cdots}}}}}.
\end{align} 
The authors [ABL], [Bu$1$] proved the following amazing result by using Schmidt Subspace Theorem.
\begin{theorem}[ABL-Bu1]\label{}
Let $\beta$ be  an integer greater than $1$. Let $(a(n))_{n=0}^\infty$ be a stammering sequence on $\{0,1,\cdots,\beta-1\}$.
Then $\sum_{n=0}^\infty \frac{a(n)}{ {\beta}^{n+1}}$ is a
transcendental number. 
Moreover  if $(a(n))_{n=0}^\infty$ be a stammering sequence on bounded positive integers, then the continued fraction $[0: a(0), a(1)\cdots,a(n) \cdots]$ is
a transcendental number, too.
\end{theorem}

 By Theorem $3.1$ and $4.1$, we will show the next theorem that includes Theorem$1.2$.
\begin{theorem}
Let $A_{\infty}=(a(n))_{n=0}^\infty$ be a $(L,k,\kappa)$-TM sequence. Let $\beta$ be an integer greater than $1$. If $(a(n))_{n=0}^\infty$ takes the values on $\{0,1,\cdots,\beta-1\}$,
then $\sum_{n=0}^\infty \frac{a(N+nl)}{ {\beta}^{n+1}}$ $($where $N \ge 0$ and  $l>0$ $)$ is a transcendental numbers unless  there  exists an integer $A$ such that
\begin{align}
\kappa(s,A+y) \equiv \kappa(1,A) s k^y    \pmod{L},
\end{align}
for all $1\le  s \le k-1$ and for all $ y \in \mathbb{N}$.
\newline
\; Moreover if $(a(n))_{n=0}^\infty$ takes the values on positive integers, then $[0: a(N), a(N+s)\cdots,a(N+nl) \cdots]$ $($where $N \ge 0$ and  $l>0$ $)$ is a transcendental numbers
unless there exists an integer $A$ such that
\begin{align}
\kappa(s,A+y) \equiv \kappa(1,A) s p^y    \pmod{L},
\end{align}
for all $1\le  s \le k-1$ and for all $ y \in \mathbb{N}$.
\end{theorem}
\begin{proof}
Let $N$ and $l>0$ be positive integers. By Theorem $3.1$, $(a(N+nl))_{n=0}^\infty$ is non periodic. Then we will prove only that $(L,k,\kappa)$-TM satisfies the condition $(2)$ of Definition $4.1$.

There exists an integer $M$ such that $k^M> 2(N+l)$. Assume $m>M$.
Since $f$ is a cyclic permutation of order $L$ and the Definition $2.2$, we see that prefix $(Ll+1)k^m$ word of $(a(n))_{n=0}^\infty$ is as follows,
\begin{align}
A_{\infty}=(a(n))_{n=0}^\infty =A_m f^{i_1}(A_m) \cdots f^{i_{Ll}}(A_m)\cdots ,
\end{align}
where $A_m$ is prefix $k^m$ word of $(a(n))_{n=0}^\infty$, $i_j (1\le j\le Ll) \in \{0,\dots, L-1 \}$.

 By $(4.4)$, we have
 \begin{align}
f^{i_{tl}}(a(n))=a(n+k^m tl ),
\end{align}
for all $0\le n\le k^m-1$ and for all $1\le t\le L$.

 By the fact $f$ is a cyclic permutation of order $L$, $(4.4)$, $(4.5)$ and Dirchlet Schubfachprinzip, then we have
\begin{align}
     (a(N+n l))_{n=0}^\infty =W_{1,m} W_{2,m} W_{3,m} W_{2,m} \cdots ,
\end{align}
where $W_{i,m}$ $(i\in\{1,2,3\})$ are finite words and
\begin{align}
    &|W_{1,m}|\le ((Ll+1)k^m -N)/l +1,
    \\&|W_{2,m}| \ge (k^m-N)/l-1,
     \\&|W_{2,m}|+|W_{3,m}|\le ((Ll+1)k^m -N)/l +1.
\end{align} 
\newline
We put $U_m:=W_{1,m}$, $V_m:=W_{2,m}W_{3,m}$ and $w:=1+\frac{1}{2 Ll+3}$.

Since $(4.7)$-$(4.9)$ and the assumption of $m$, then we get
\begin{eqnarray}
      &\lceil (w-1)|V_m|\rceil  =\lceil \frac{1}{ 2 Ll+3}(|W_{2,m}|+|W_{3,m}|)\rceil  \le \nonumber \\%
      &\frac{1}{ 2 Ll+3}((Ll+1)k^m-N+l)/l\le \frac{k^m}{2 l} \; < |W_{2,m}|.
\end{eqnarray}
From $(4.10)$, we see that  $(a(n))_{n=0}^\infty$ satisfies the condition $(A)$.

Furthermore we have,
\begin{eqnarray}
     &|U_m|/|V_m|=|W_{1,m}|/|W_{2,m}W_{3,m}|\le \nonumber \\%
&((Ll+1)k^m-N+l)/l \times l /( k^m-N-l) \le 2 Ll+3.
\end{eqnarray}
From $(4.11)$, we see that $(a(n))_{n=0}^\infty$ satisfies the condition $(B)$.

Obviously, $(V_m)_{m\ge 1}$ satisfies the condition $(C)$.
This completes the proof.
\hfill
$\square$
\end{proof}

\begin{Cor}
Let $(a(n))_{n=0}^\infty$ be a $(L,k,\kappa)$-TM sequene. Let $\beta$ be an integer greater than $1$. If $(a(n))_{n=0}^\infty$ takes the values on $\{0,1,\cdots,\beta-1\}$
,then the generating function $f(z):=\sum_{n=0}^\infty \frac{a(N+nl)}{ z^{n+1}}$ $($where $N \ge 0$ and  $l>0$ $)$ are transcendental over $\mathbb{C}(z)$ unless there exists an integer $A$ such that
\begin{align}
\kappa(s,A+y) \equiv \kappa(1,A) s k^y    \pmod{L},
\end{align}
for all $1\le  s \le k-1$ and for all $ y \in \mathbb{N}$.
\end{Cor}
\begin{proof}
We assume $f(z)$ is algebraic over $\mathbb{C}(z)$.
By the fact $f(z)$ is algebraic over $\mathbb{Q}(z)$ if and only if  $f(z)$ is algebraic over $\mathbb{C}(z)$ $($see Remark of Theorem$1.2$ in $[$N$]$ $)$,
then there exist $c_i(z) \in \mathbb{Q}(z) $  $( 0 \le i \le n)$ such that $c_n(z) c_0(z) \neq 0$, $c_i(z)$ $(0 \le i \le n)$ are coprime and satisfy following equation
\begin{align}
      c_n(z) f^n(z)+c_{n-1}(z)f^{n-1}(z)+ \cdots + c_0(z)=0.
\end{align}
From Theorem $4.2$, we see that $f(\frac{1}{\beta})$ is a transcendental number. By the fact mentioned above and $(4.13)$, we get
$c_i(\frac{1}{\beta}) =0$  $($ for all $0 \le i \le n)$. This contradics the assumption that $c_i(z)$ $( 0 \le i \le n)$ are coprime.
\hfill
$\square$
\end{proof}
\begin{center}\section{The $k$-automatic generalized Thue-Morse sequences and some results}\end{center}
First we  introduce some definitions.
\begin{Def}
Let $\alpha$ be an irrational real number. The irrationality exponent $\mu (\alpha)$ of  $\alpha$ is the supremum of the real numbers
$\mu$ such that the inequality
\begin{align}
     | \alpha - \frac{p}{q} | < \frac{1}{q^{\mu}},
\end{align}
has infinitely many solutions in non zero integers $p$ and $q$.
\end{Def}
\begin{Def}
The $k$-kernel of $(a(n))_{n=0}^\infty$ is the set of all subsequences of the form
$(a(k^e n+j))_{n=0}^\infty$ where $e\ge 0$ and $0\le j\le k^e-1$.
\end{Def}
\begin{Def}
We say $(a(n))_{n=0}^\infty$ is a $k$-automatic sequence if the $k$-kernel of  $(a(n))_{n=0}^\infty$
is the finite set.
\end{Def}
\begin{Def}
We say $\sum_{n=0}^\infty a(n) z^n \in\mathbb{C}[[x]]$ is  a $k$-automatic power series if
$(a(n))_{n=0}^\infty$ is a $k$-automatic sequence.  
\end{Def}
\begin{Def}
We say  $(L,k,\kappa)$-TM sequence is $n$-period if there  exist  non negative integers $N$ and $t(0<t)$ such that 
\begin{align}\label{def:period}
     \kappa (s,n)=\kappa (s,n+t),
\end{align} 
for all $1\le s\le k-1$ and for all $n \ge N$ .
\end{Def}
Now we introduce the following two results.
\begin{theorem}[AC]\label{}
Let $\beta$ be  an integer greater than $1$. Let $(a(n))_{n=0}^\infty$ is a non-periodic $k$-automatic sequence on $\{0,1,\cdots,\beta-1\}$.
Then $\mu(\sum_{n=0}^\infty \frac{a(n)}{ {\beta}^{n+1}})$ is finite.
\end{theorem}
This theorem  is obtained in [AC].
\begin{theorem}[Be]\label{}
Let $f(z) \in \mathbb{Q}[[z]]$ be a $k$-automatic power series. Let $0<R<1$. If  $f(z)$ is transcendental over $\mathbb{Q}(z)$, then $f(\alpha)$
is transcendental for all but finitely many algebraic numbers $\alpha$ with $|\alpha| \le  R$.
\end{theorem}
This theorem  is obtained in [Bec].

Now we consider $(L,k,\kappa)$-TM sequence of the necessary-sufficient condition that it is a $k$-automatic sequence.
\begin{Prop}
A $(L,k,\kappa)$-TM sequence is $n$-period if and only if it is a $k$-automatic sequence.
\end{Prop}
\begin{proof}
Assume without loss of generality that $A_{\infty}=(a(n))_{n=0}^\infty$ is a $(L,k,\kappa)$-TM sequence with $a_j=\exp \frac{2\pi \sqrt{-1} j}{L}$ $($where $0\le \; j \le L-1$ $)$.

We assume $(a(n))_{n=0}^\infty$ is a $k$-automatic sequence. Since the $k$-kernel of  $(a(n))_{n=0}^\infty$
is finite set, we see that there exist integers $e$ and  $0<t $ such shat 
\begin{align}
        a(k^e n)=a(k^{e+t} n)  \qquad (\forall n \ge 0).
 \end{align}
Let $s$ be any integer in $\{1,2,\cdots,k-1\}$ and $y$ any integer in $\mathbb{N}$. By Lemma $2.3$ with $(3.10)$ and $(5.3)$ with substituting $sk^y$ for $n$, we have
\begin{align}
         \exp \frac{2\pi \sqrt{-1}\kappa (s,e+y)} {L}=  a(k^e s k^y )=a(k^{e+t} s k^y) =\exp \frac{2\pi \sqrt{-1}\kappa (s,e+y+t)} {L}.
\end{align}
Since the definition of $(L,k,\kappa)$-TM sequence and $(5.4)$, then $(a(n))_{n=0}^\infty$ is $n$-period.

 We will show converse. If  $(L,k,\kappa)$-TM sequene $A_{\infty}=(a(n))_{n=0}^\infty$ is $n$-period,
then there  exist  non negative integers $e$ and $0<t $ such that 
\begin{align}\label{def:period}
      \kappa (s,e+n)=  \kappa (s,e+n+t) \qquad (\forall n \ge e \;and\; \; 1\le \forall s\le k-1 ). 
\end{align} 
Let $l$ be any integer greater than $t-1$ 
and $(a(k^{e+l} n+j))_{n=0}^\infty$ $($where $0\le j \le k^{e+l}-1$ $)$ any sequence in $k$-kernel of  $(a(n))_{n=0}^\infty$.

From Lemma $2.3$ with $(3.10)$, we get  
\begin{align}
           a(k^{e+l} n+j)=a(k^{e+l} n)a(j).
\end{align}
By the fact $(a(n))_{n=0}^\infty$ takes on only finitely many values, then $a(j)$
takes on only finitely many values, too.

Let the $k$-adic expansion of
$n$ as follows,
\begin{align}
          n=\sum_{q=1}^{N(n)} s_q k^{w(j)} \qquad where \;1\le s_q\le k-1, w(q+1)>w(q)\ge 0.  
\end{align}
Let $l(t)\equiv l \;\pmod {t}$ where $0\le l(t) \le t-1$. By Lemma $2.3$ with $(3.10)$, we have
\begin{align}
           a(k^{e+l} n)=a(\sum_{q=1}^{N(n)} s_q k^{w(q)+e+l} )=
           \prod_{q=1}^{N(n)}a(s_q k^{w(q)+e+l}).
\end{align}
From $(5.7)$, $(5.8)$, and Lemma $2.3$ with $(3.10)$, we get
\begin{align}
           &a(k^{e+l} n)= \prod_{q=1}^{N(n)}a(s_q k^{w(q)+e+l})=
           \prod_{q=1}^{N(n)}a(s_q k^{{w(q)+e+l(t)} })\nonumber \\
           &=a(\sum_{q=1}^{N(n)} s_q k^{w(q)+e+l(t)} )=
           a(k^{e+l(t)} n).
\end{align}
By the fact $a(j)$ takes on only finitely many values, $(5.6)$ and $(5.9)$, 
then the $k$-kernel of  $(a(n))_{n=0}^\infty$ is finite set.
This completes the proof.
\hfill
$\square$
\end{proof}

\begin{theorem}
Let $(a(n))_{n=0}^\infty$ be a $(L,k,\kappa)$-TM. Let $\beta$ be an integer greater than $1$. If $(a(n))_{n=0}^\infty$ takes on $\{0,1,\cdots,\beta-1\}$
and  $n$-period, then $\mu(\sum_{n=0}^\infty \frac{a(N+nl)}{ {\beta}^{n+1}})$ $($where $N \ge 0$ and  $l>0$ $)$ is finite unless there exists an integer $A$ such that
\begin{align}
\kappa(s,A+y) \equiv \kappa(1,A) s k^y    \pmod{L},
\end{align}
for all $1\le  s \le k-1$ and for all $y \in \mathbb{N}$.
\end{theorem}
\begin{proof}
By previous proposition, $(a(n))_{n=0}^\infty$ is a $k$-automatic sequence. By the fact if $(a(n))_{n=0}^\infty$ is $k$-automatic, then
$(a(N+nl)_{n=0}^\infty$ $($where $N \ge 0$ and  $l>0$ $)$ are $k$-automatic, too. $($see Theorem $2.6$ in $[$AS1$]$ $)$.
From Theorem $4.2$ and $5.1$, we see that $\mu(\sum_{n=0}^\infty \frac{a(N+nl)}{ {\beta}^{n+1}})$ is finite.
\hfill
$\square$
\end{proof} 

\begin{theorem}Let $(a(n))_{n=0}^\infty$ be a $(L,k,\kappa)$-TM. Let $\beta$ be an integer greater than $1$.
Let $f(z):= \sum_{n=0}^\infty \frac{a(N+nl)}{ z^{n+1}}$ $($where $N \ge 0$ and  $l>0$ $)$. Let $0<R<1$. If $(a(n))_{n=0}^\infty$ takes on $\{0,1,\cdots,\beta-1\}$
and  $n$-period, then $f(\alpha)$
is transcendental number for all but finitely many algebraic numbers $\alpha$ with $|\alpha| \le  R$ unless there exists an integer $A$ such that
\begin{align}
\kappa(s,A+y) \equiv \kappa(1,A) s k^y    \pmod{L},
\end{align}
for all $1\le  s \le k-1$ and for all $y \in \mathbb{N}$.
\end{theorem}
\begin{proof}
By Corollary $4.1$, $f(z)$ is transcendental over $\mathbb{Q}(z)$.
From Proposition $5.1$, we see that $(a(N+nl))_{n=0}^\infty$ $($where $N \ge 0$ and  $l>0$ $)$ are $k$-automatic sequences, too.
Then $f(z)$ is a $k$-automatic power series. Since Theorem $5.2$, then $f(\alpha)$
is transcendental for all but finitely many algebraic numbers $\alpha$ with $|\alpha| \le  R$.
This completes the proof.
\hfill
$\square$
\end{proof} 

\bibliographystyle{model3a-num-names}

\bibliography{<your-bib-database>}

\end{document}